\documentclass[10pt]{amsart}

\usepackage{amsmath,amssymb,graphicx,epsfig}
\usepackage{amsmath}
\usepackage{amssymb}
\usepackage{framed}
\usepackage{subcaption}
\usepackage{graphics}
\usepackage{amsmath,amssymb,graphicx,epsfig}
\usepackage{color}
\usepackage[colorlinks=true,linkcolor=blue]{hyperref} 

\newtheorem{theorem}{{\bf Theorem}}[section]
\newtheorem{lemma}[theorem]{{\bf Lemma}}
\newtheorem{cor}[theorem]{{\bf Corollary}}

\numberwithin{equation}{section}

\title[Inverse scattering in an asymptotically flat multilayer domain]
{Inverse scattering in an asymptotically flat multilayer domain}

\date{\today}

\author{Michel Cristofol$^1$}
\author{Yves Dermenjian$^2$}
\author{Hiroshi Isozaki$^3$}
\author{Olivier Poisson$^4$}

\address
{$^1$Aix-Marseille Universit{\'e}, CNRS, I2M, Marseille, France, michel.cristofol@univ-amu.fr}

\address
{$^2$Aix-Marseille Universit{\'e}, CNRS, I2M, Marseille, France, yves.dermenjian@univ-amu.fr}
\address
{$^3$Department of Pure and Applied Sciences, University of Tsukuba, Tsukuba, Tennoudai, 1-1-1, 305-0032, Japan, isozakih@math.tsukuba.ac.jp}
\address
{$^4$Aix-Marseille Universit{\'e}, CNRS, I2M, Marseille, France, olivier.poisson@univ-amu.fr}
\begin{document}

\baselineskip 14pt
\maketitle

\begin{abstract}
We consider a scattering problem for a wave equation $\partial_t^2 u = \frac{1}{\sqrt{g}}\partial_i(\sqrt{g}g^{ij}\partial_j)u$ in a multilayer domain $\Omega \subset {\bf R}^{n+1}_x = {\bf R}^n_y \times {\bf R}^1_{x^{n+1}}$ of the form $\Omega = \mathcal K \cup \Omega_1 \cup \cdots \cup \Omega_N$, where $\mathcal K$ is a bounded open set and $\Omega_k$ is asymptotically equal to a slab domain ${\bf R}^n \times (c_k,c_k +  d_k)$ as $|y| \to \infty$. Assuming that $\partial_x^{\alpha}\big(g_{ij}(x) - \delta_{ij}\big) = O(|x|^{-|\alpha| - \delta_0}), \ \delta_0 > 1, \forall \alpha$, we show that $\Omega$ and $g^{ij}$ are determined by one diagonal component $S_{11}(\lambda)$, for all energies, of the S-matrix associated with the slab $\Omega_1$, provided $\Omega_1$ is flat: $\Omega_1 \cap \{|y| > R\} = \{|y| > R\} \times (c_1, c_1+d_1)$ for some constants $c_1, d_1, R > 0$,  and the metric is Euclidean on $\Omega_1\cap \{|y| > R\}$.  
\end{abstract}

\section{Introduction}
\label{SectionIntro}
Consider a domain $\Omega \subset {\bf R}^{n+1}_x = {\bf R}^n_y \times {\bf R}^1_{x_{n+1}}$, $x = (y,x_{n+1})$,  as in Figure \ref{Fig.1} and $\Delta_y + (\partial_{n+1})^2$ in $\Omega$.
 Assume that there exists a constant $C > 0$ such that $\Omega \subset \{|x_{n+1}| < C\}$ and there exist $f_{\pm}(y) \in C^{\infty}({\bf R}^n)$ and constants $C_- < C_+$ such that  the boundary $\partial \Omega = \Gamma_+ \cup \Gamma_-$ is written as $x_{n+1} = f_{\pm}(y)$ with $\partial_y^{\alpha}(f_{\pm}(y) - C_{\pm}) = O( |y| ^{-|\alpha| - \delta_0})$ as $|y| \to \infty$ for a constant $\delta_0 > 1$, $\forall \alpha$.
 Then, without loss of generality, we can assume that there exists $R_0 > 0$ such that
 $\Omega \cap \{|y| > R_0\}$ is diffeomorphic to $\{|y| > R_0\} \times (0,d)$, $d= C_+ - C_-$,  equipped
 with the Laplacian
\begin{equation}
- H :=  \Delta_G = \frac{1}{\sqrt{g}}\partial_i\big(\sqrt{g} g^{ij}\partial_j\big),
\nonumber
\end{equation}
 where $G = ( g_{ij}) = (g^{ij})^{-1}$ is a Riemannian metric on $\Omega$, $\partial_i = \partial/\partial y_i$,
 $i = 1, \dots, n$, $\partial_{n+1} = \partial_{x_{n+1}}$,  and
\begin{equation}
\partial_x^{\alpha}\big(g_{ij}(x) - \delta_{ij}\big) = O(|x|^{-|\alpha| - \delta_0}), \quad \delta_0 > 1, \quad \forall \alpha. 
\label{Condgij(x)-edltaij}
\end{equation}
 We call such $\Omega$ an asymptotically flat slab, 
 and $\textcolor{blue}{d}$ the thickness of the slab.
 In the case that $\Omega_0 = {\bf R}^n\times (0,1)$ equipped with the Euclidean metric
 and the Neumann boundary condition on $x_{n+1} = 0,1$, we call it a model slab.
 

 \begin{figure}[]
\setlength{\unitlength}{0.6cm}
\centering
\begin{subfigure}{0.49\textwidth}
\begin{picture}(9,6)
\put(1,2){\vector(1,0){8}}
\put(5,-1){\vector(0,1){6}}
\put(5,5.3){${\bf R}^1_{x_{n+1}}$}
\put(9.5, 2){${\bf R}^n_y$}
\bezier{200}(3,3)(5,4)(7,3)
\put(3,3){\line(-1,0){2}}
\put(7,3){\line(1,0){2}}
\put(7,1){\line(1,0){2}}
\put(3,1){\line(-1,0){2}}
\bezier{200}(3,1)(5,0)(7,1)
\put(5.5,3.7){$\Gamma_+$}
\put(5.5,0){$\Gamma_-$}
\put(4.2, 1.4){$\Omega$}
\end{picture}
\vskip 6mm
\hskip 15mm
\caption{Asymptotically flat slab}
\label{Fig.1A}
\vskip 5mm
\end{subfigure}
\hfill
\begin{subfigure}[r]{0.49\textwidth}
\begin{picture}(9,6)
\put(0.5,4){\line(1,0){2.5}}
\bezier{200}(3,4)(5,4.5)(7,4)
\put(7,4){\line(1,0){2}}
\bezier{200}(3.5,3)(4,2)(3.5,1)
\put(3.5,3){\line(-1,0){3}}
\put(6,3){\line(1,0){3}}
\put(6,1){\line(1,0){3}}
\put(3.5,1){\line(-1,0){3}}
\bezier{200}(6,3)(5.5,2)(6,1)
\put(0.5,0){\line(1,0){2.5}}
\put(4.5,1.8){$\Omega$}
\bezier{200}(3,0)(5,-0.5)(7,0)
\put(7,0){\line(1,0){2}}
\end{picture}
\vskip 1mm
\hskip 5mm
\caption{Composed slab}
\label{Fig.1B}
\end{subfigure}
\caption{Slab}
\label{Fig.1}
\hfill
\end{figure}

 More generally, we consider a connected smooth open subset $\Omega \subset {\bf R}^{n+1}$
 such that
\begin{equation}
\label{eq-defM:1}
\overline{\Omega} = \overline{\mathcal K}
 \cup \overline{\Omega_1} \cup \cdots \cup \overline{\Omega_N}, 
\end{equation}
 where $\mathcal K$ is a bounded open set, $\Omega_m$, $m = 1, \dots, N$ are asymptotically flat slabs
 (see Figure \ref{Fig.1}(B)) satisfying (\ref{Condgij(x)-edltaij}),
 where the metric on $\Omega_m$, denoted by $(g_{m,ij})$ may be different for $m = 1, 2, \dots, N$. 
We assume that $\Omega_i \cap \Omega_j = \emptyset$ if $i \neq j$, which means that $\Omega_i$'s are almost parallel near infinity.
We are dealing with an idealized modelisation of concrete situations as can be found in geophysics with aquifer-systems. 
See e.g.  \cite{HHSCCL21} p. 9, \cite{SXNHMS25} p. 34.

 The problem we address is the inverse scattering on such $\Omega$ for $H$,
 the Laplacian $-\Delta_G$ with Neumann boundary condition\footnote{We can also add a potential term. }.
 Namely, we are interested in the following problem:
\begin{itemize}
\item Recovery of the topology and the metric on $\Omega$ from the knowledge of the S-matrix.
\end{itemize}

Let us briefly recall the idea of the $S-$matrix, which is a representation of the scattering operator. Given a Hamiltonian $H$ and the wave equation $\partial_t^2 u = - Hu$, we transform it into a first order system $i \partial_tv = \mathcal Hv$ and consider the associated evolution operator 
$\mathcal U(t) : f \to e^{-it\mathcal H}f$.
The scattering operator links the behaviour of $\mathcal U(t)f$ as $t\to -\infty$ to that as $t\to +\infty.$ Its representation by $S$-matrix is believed to contain knowledge equivalent to the system in question. For more details, see  for example \S 3.1 of
 \cite{Iso:1}.

A similar problem for the case of asymptotically cylindrical manifolds was studied in \cite{IKL10}. In this case, each end $\Omega_i$ $(i = 1, \dots, N)$ is diffeomorphic to $[R_i,\infty)\times M_i$, where $M_i$ is a compact manifold of dim. $n$ (with or without boundary). In our following  argument,  we  assume that $n \geq 2$ to avoid the case of  cylindrical end. 
 
We study the limiting absorption principle (LAP), and the asymptotic behavior of solutions of Helmholtz equation at infinity, from which the S-matrix  is derived.
The inverse scattering procedure will then be as follows. We are given two such asymptotically flat slabs $\Omega^{(1)}$ and $\Omega^{(2)}$ as in (\ref{eq-defM:1}).  
The associated S-matrix is an operator-valued matrix $\mathcal S(\lambda) = \big(\mathcal S_{ij}(\lambda)\big)$, $1 \leq i, j \leq N$, where each entry $\mathcal S_{ij}(\lambda)$ is a bounded operator from $L^2(S^{n-1}\times (c_i, c_i + d_i))$ to $L^2(S^{n-1}\times (c_j, c_j + d_j))$, where $d_i$ is the thickness of $\Omega_i$. Our aim is to show the following fact: The whole domain $\Omega$ is determined by one diagonal entry of the S-matrix, i.e.
\begin{theorem}
\label{MainTheorem}
Given two asymptotically flat slabs $\Omega^{(1)}$ and $\Omega^{(2)}$, assume that $\mathcal S_{11}^{(1)}(\lambda) = \mathcal S^{(2)}_{11}(\lambda)$ for all $\lambda \in (0,\infty)\setminus \cup_{i=1,2}(\mathcal E(H^{(i)})\cup\mathcal E(H_{\emptyset}^{(i)}))$, and that $\Omega^{(1)}_1\cap \{|y| > R\}$ and $\Omega^{(2)}_1\cap \{|y| > R\}$ are flat and isometric in the sense of Euclidean metric for some $R > 0$. Then, $\Omega^{(1)}$ and $\Omega^{(2)}$ are isometric. 
\end{theorem}

For the meaning of the notation $\mathcal E(H^{(i)})\cup\mathcal E(H_{\emptyset}^{(i)})$, see (\ref{DefinemathcalEH}) and \S \ref{SectionFromScatteringtoInterior}.\footnote{$H_{\emptyset}$ is the Neumann Laplacian $- \Delta_G$ on $\Omega\setminus \overline{(\Omega_1\cap \{|y| > R\})}$.}

The paper is organized as follows. In \S 2 we describe the forward problem for the model space. More precisely, we construct the free and pertubed spectral representations and we define the S-matrix. Then we derive the analytic continuation of the scattering amplitude. In \S 3, we begin to reconstruct the domain $ \Omega$ by studying the interior problem on the basis of  information obtained from the S matrix.  We focus on  the method involving the source to solution map. Finally in \S 4 we complete the proof.

\section{Forward problem}
\subsection{Spectral theory for the model space}
 Let us make the definition of $\Omega$  in (\ref{eq-defM:1}) more precise. In ${\bf R}^{n+1}$, we are given a domain 
 $\Omega = \mathcal K \cup \Omega_1 \cup \cdots \cup \Omega_N$ satisfying the following conditions: 
 Each $\Omega_j$ is an open set of the form $\{(y,x_{n+1})\, ; \,
 |y| > R_0$, $c_j < x_{n+1} < c_j + d_j\}$ which is diffeomorphic to $\{|y| > R_0\} \times (0,d_j)$
 where $R_0=2^{\ell_0}$ for some $\ell_0 > 0$, which we simply denote as
 $\Omega_j = \{|y| > R_0\} \times (0,d_j)$.
 We assume that 
\begin{equation}
\overline{\Omega_i} \cap \overline{\Omega_j} = \emptyset, \quad 
{\it if} \quad i \neq j.
\nonumber
\end{equation}
 Moreover, $\mathcal K = \Omega \setminus \big(\overline{\Omega_1} \cup \cdots \cup \overline{\Omega_N}\big)$
 is a bounded open set, and  $\overline{\Omega}$ is a $C^{\infty}$ manifold with boundary.
 Assume that $\Omega$ is equipped with a Riemannian metric satisfying (\ref{Condgij(x)-edltaij}) on each $\Omega_j$.

In the following, $X$ denotes a point in $\Omega$, while $x = (y,x_{n+1})$ denotes a corresponding point in the model slab or in $\Omega_j$
\footnote{Our dissussions and notations below are parallel to the corresponding ones in \cite{IKL10}, where on each cylinder $(0,\infty) \times M$, $M$ being a manifold of dimension $n$, $y$ varies over $(0,\infty)$ and $\omega$ over $M$. In our case of slab, 
$y$ varies over ${\bf R}^n$ and $x_{n+1}$ over $(0,d)$.}. Letting ${\rm{dist}}\, (X,X')$ be the Riemannian distance of $X, X' \in\Omega$, and fixing $X_0 \in \Omega$ arbitrarily, we define for $s \in {\bf R}$
\begin{equation}
L^{2,s} \ni f \Longleftrightarrow \|f\|_s^2 = \int_{\Omega}
\big(1 + d(X,X_0)\big)^{2s}|f(X)|^2 d\Omega_X < \infty,
\nonumber
\end{equation}
where $d\Omega_X$ is the volume element of $\Omega$. 
We also use the Agmon-H\"{o}rmander space $\mathcal B, \mathcal B^{\ast}$, which are the Besov type spaces
 associated with $\Omega$ defined as follows:
For $f \in L^2_{loc}(\Omega_j)$, we put
\begin{equation}
\|f\|_{\mathcal B(\Omega_j)} = \sum_{\ell=\ell_0}^{\infty}
2^{\ell/2}\Big(\int_{2^{\ell} < |y| < 2^{\ell+1}}\|f(y,\cdot)\|^2_{L^2((0,d_j))}dy\Big)^{1/2},
\nonumber
\end{equation}
\begin{equation}
\|f\|_{\mathcal B(\Omega)} = \|f\|_{L^2(\mathcal K)} + 
\sum_{j=1}^N\|f\|_{\mathcal B(\Omega_j)}.
\nonumber
\end{equation}
 with $R_0=2^{\ell_0}$.
 The norm of ${\mathcal B}^{\ast}(\Omega)$ is defined as follows:
\begin{equation}
 \|u\|^2_{{\mathcal B}^{\ast}} = \|u\|^2_{L^2(\mathcal K)} + 
 \sum_{j=1}^N \sup_{R>R_0} \frac{1}{R}\int_{\Omega_j \cap \{|y| < R\}} |u(X)|^2 d\Omega_X < \infty.
 \nonumber
\end{equation}
We define also the following relation of equivalence on $\mathcal B^{\ast}$:
\begin{equation}
 u \simeq v \Longleftrightarrow  \lim_{R\to\infty}
 \sum_{j=1}^N \frac{1}{R} \int_{\Omega_j \cap \{|y| < R\}} |u(X)-v(X)|^2 d\Omega_X = 0.
\label{ExpansioninBast(Omega)}
\end{equation}
 The spaces $\mathcal B(\Omega_j)$, $\mathcal B^{\ast}(\Omega_j)$ and the relation $u \simeq v$ on $\Omega_j$
 are defined similarly. We often omit $\Omega$ or $\Omega_j$ in the notation of $\mathcal B, \mathcal B^{\ast}$ spaces.
 The following inclusion relations hold: For $s > 1/2$,
\begin{equation}
 L^{2,s} \subset \mathcal B \subset L^{2,1/2} \subset L^2 \subset L^{2,-1/2} \subset \mathcal B^{\ast} \subset L^{2,-s}.
 \nonumber
\end{equation}

 First let us consider the model slab ${\bf R}^n\times (0,1)$. We know that for any $\lambda > 0$, the limit $(- \Delta_y - \lambda \mp i0)^{-1}$ exists as a bounded operator from $L^{2,s}({\bf R}^n)$ to $L^{2,-s}({\bf R}^n)$ with $s>1/2$, and  from $\mathcal B({\bf R}^n)$ to $\mathcal B^{\ast}({\bf R}^n)$. 
 Let $ - (\partial_{n+1})^2$ be $- (\partial_{x_{n+1}})^2$ on $(0,1)$ with Neumann boundary condition. 
To prove the existence of these limits, i.e. the limiting absorption principle, abbreviated as LAP, for $- \Delta_y - (\partial_{n+1})^2$ on ${\bf R}^n\times (0,1)$, we choose an orthnormal basis of $- (\partial_{n+1})^2$ which enables us to prove these limits for each eigenspace associated to these basis.
 Expanding by the orthonormal basis of $-(\partial_{n+1})^2$ one can prove the limiting absorption principle, abbreviated as LAP,
 for $-\Delta_y - (\partial_{n+1})^2$ on ${\bf R}^n\times (0,1)$.

 Consider the slab $\Omega_j$, and let $\chi_j \in C^{\infty}(\Omega;[0,d_j])$ be such that $\chi_j = 1$
 on $\Omega_j\cap\{|y| > R_0 + 2\}$, $\chi_j = 0$ on $\Omega_j\cap\{|y| < R_0 + 1\}$, and also
 $\chi_j = 0$ on $\Omega\setminus\Omega_j$. Define $\chi_0 = 1 - \sum_{j=1}^N\chi_j$. Then, $\{\chi_j\}_{j=0}^N$
 is a partition of unity on $\Omega$.
 The conormal differentiation at the boundary with respect to the metric $G$ is denoted by $\partial_{\nu}$,
 and that with  respect to the Euclidean metric is denoted by $\partial_{\nu^{(0)}}$.  We set, for $R\ge R_0$,
\begin{equation}
\partial \Omega_j(R) = \partial\Omega_j \cap \{|y| > R\} 
= \{(y,0), (y,d_j) \in \partial\Omega_j\, ; \, |y| > R\}.
\nonumber
\end{equation}
 $H^{m,s}(\Omega_j)$ is the weighted Sobolev space defined by $u \in H^{m,s}(\Omega_j) \Longleftrightarrow
 (1 + |y|)^su  \in H^m(\Omega_j)$,  and $H^{m,s}(\partial\Omega_j)$ is defined similarly with weight
 $(1 + |y|)^s$.
 When $s = 0$, $H^{m,0}$ is denoted by $H^m$. 
 The following lemma can be proven in the same way as Lemmas 3.1 and 3.2 in \cite{IKL10}.
\begin{lemma}
\label{ExtensionLemma}(1) There exists a real function $w(x) \in C^{\infty}(\Omega_j)$ such that 
\begin{equation}
\left\{
\begin{split}
& \partial_{\nu}w(x) = 0 \quad {\rm on} \quad \partial\Omega_j(R), \\
& w(x) = |y| + O(|y|^{-\delta_0}), \quad {\rm as}\quad |y| \to \infty,
\end{split}
\right.
\nonumber
\end{equation}
 where $\delta_0$ is given in (\ref{Condgij(x)-edltaij}).\\
\noindent
(2) Let $R>R_0+1$. There exists an operator of extension $\widetilde{\mathcal E}_j$
 such that for $m \geq 1/2$ and $\psi \in H^m(\partial\Omega_j(R))$
\begin{equation}
\partial_{\nu}\widetilde{\mathcal E}_j \psi = 
\psi \quad {\rm on} \quad \partial\Omega_j(R), 
\nonumber
\end{equation}
\begin{equation}
{\rm supp}\, (\widetilde{\mathcal E}_j\psi) \subset 
\Omega_j \cap \{|y| > R-1\}.
\nonumber
\end{equation}
Moreover, for $m \geq 1/2$ and $s \geq 0$, it satisfies
\footnote{For Banach spaces $\mathcal X$ and $\mathcal Y$, ${\bf B}(\mathcal X ; \mathcal Y)$ is the set of all bounded
 operators from $\mathcal X$ to $\mathcal Y$.}
\begin{equation}
\widetilde{\mathcal E}_j \in {\bf B}(H^{m,s}(\partial\Omega_j(R));H^{m+3/2,s}(\Omega_j)).
\nonumber
\end{equation}
\end{lemma}

We then have for $u \in H^2(\Omega_j)$ satisfying 
$\partial_{\nu_j^{(0)}}u = 0$ on $\partial\Omega_j(R)$,
\begin{equation}
\partial_{\nu}(\chi_j u) = w(x)^{-\delta_0}B_ju, \quad 
{\rm on} \quad \partial\Omega_j(R),
\nonumber
\end{equation}
\begin{equation}
B_j = w(x)^{\delta_0}\big(\chi_j(\partial_{\nu} - \partial_{\nu^{(0)}}) + (\partial_{\nu}\chi_j)\big).
\label{Bjdefine}
\end{equation}
We put
\begin{equation}
\mathcal E_j = w(x)^{-\delta_0}\widetilde{\mathcal E}_j.
\nonumber
\end{equation}
Letting $G_j^{(0)}$ be the Euclidean metric on $\Omega_j$, 
we also put
\begin{equation}
\mathcal V_j(z) = [ - \Delta_G, \chi_j] + \chi_j(\Delta_{G_j^{(0)}} - \Delta_G) + (\Delta_G + z)\mathcal E_jB_j.
\label{DefineVj(z)}
\end{equation}
Let $H_j^{(0)} = - \Delta_y - (\partial_{n+1})^2$ in $\Omega_j$ with Neumann boundary condition on the boundary and $R_j^{(0)}(z) = (H_j^{(0)} - z)^{-1}$. Finally, let $H$ be the Laplacian $-\Delta_G$ on $\Omega$ with
 Neumann boundary condition on $\partial\Omega$ and $R(z) = (H - z)^{-1}$. 
 Then, as in Lemma 3.3 in \cite{IKL10}, we have
\begin{lemma}
\label{LemmaResolventEquation}
 Let $R\ge R_0+4$.
 Let $\tilde\chi_j \in C^{\infty}(\Omega)$ be such that 
 $\tilde\chi_j = 1$ on $\Omega_j\cap \{|y| > R-1\}$ and $\tilde\chi_j = 0$ outside
 $\Omega_j \cap \{|y| > R-2\}$. 
 Then for $z \not\in {\bf R}$, the following resolvent equations hold: 
\begin{equation}
 R(z) \tilde \chi_j =  \Big(\chi_j - \mathcal E_jB_j - R(z)\mathcal V_j(z)\Big)R_j^{(0)}(z)\tilde\chi_j,
\label{chijR(z)=tildechijJj-1R(0)j(z)1}
\end{equation}
\begin{equation}
 \tilde \chi_j R(z) = \tilde\chi_jJ_j^{-1}R_j^{(0)}(z)J_j\Big(\chi_j -
  (\mathcal E_jB_j)^{\ast} - \mathcal V_j(\overline{z})^{\ast}R(z)\Big), 
 \label{chijR(z)=tildechijJj-1R(0)j(z)}
\end{equation}
 where $J_j = (\det G/\det G_j^{(0)})^{1/2}$, and the adjoint $\ast$ is taken with respect to the inner product
 of $L^2(\Omega)$ with volume element from the metric $G$. Moreover, $R_j^{(0)}(z)J_j(\mathcal E_jB_j)^{\ast}$ and $R_j^{(0)}(z)J_j\mathcal V_j(\overline z)^{\ast}R(z)$ are compact on $L^2(\Omega)$.
\end{lemma}

 Let $\Delta_G$ be the Laplacian on $\Omega$, and $H = - \Delta_G$ with Neumann boundary condition
 on $\partial\Omega$. By the method of singular sequence, we can prove
\begin{lemma}
$\sigma(H) = [0,\infty).$
\end{lemma}
 Using the resolvent equations (\ref{chijR(z)=tildechijJj-1R(0)j(z)1}), (\ref{chijR(z)=tildechijJj-1R(0)j(z)}) in Lemma \ref{LemmaResolventEquation}, one can prove  LAP for $H$.
 In fact, letting $- (\partial_{j,n+1})^2$ be the Neumann Laplacian on the interval $(0, d_j)$,
 we first prove LAP for $H_j^{(0)}$ on $\Omega_j$ using LAP for $- \Delta_y - (\partial_{j,n+1})^2$,
 and then prove LAP for $H$ by  perturbation arguments.

 We define the set of thresholds for $H$ by
\begin{equation}
\mathcal T(H) = \cup_{j=1}^N\sigma_p(- (\partial_{j,n+1})^2),
\nonumber
\end{equation}
and the set of exceptional points by
\begin{equation}
\mathcal E(H) = \mathcal T(H) \cup \sigma_p(H).
\label{DefinemathcalEH}
\end{equation}
\begin{theorem} 
\label{LAPforH}(1) 
 $\mathcal T(H)$ and $\mathcal E(H)\setminus \mathcal T(H)$ are  discrete sets, and their  possible accumulation
 points belong to  $\mathcal T(H)$. \\
\noindent
(2) For $\lambda \in (0,\infty)\setminus \mathcal E(H)$ and $f \in \mathcal B$, there exists a limit $R(\lambda \pm i0)f = \lim_{\epsilon \to 0}R(\lambda \pm i\epsilon)f$ in the sense 
$\lim_{\epsilon \to 0}\big(R(\lambda \pm i\epsilon)f, g\big), \ \forall g \in \mathcal B(\Omega)$.  If $f \in L^{2,s}(\Omega)$ for some $s > 1/2$, we have $R(\lambda \pm i0)f = {\rm s-lim}_{\epsilon \to 0}R(\lambda \pm i\epsilon) f$ in $L^{2,-s}(\Omega)$. \\
\noindent
(3) For any compact interval $I \subset (0,\infty)\setminus \mathcal E(H)$, there exists a constant $C>0$ such that 
$$
\|R(\lambda\pm i0)f\|_{\mathcal B^{\ast}} \leq C\|f\|_{\mathcal B}, \quad \forall \lambda \in I,
$$
and for any $s > 1/2$, there exists a constant $C_s>0$ such that 
$$
\|R(\lambda\pm i0)f\|_{-s} \leq C_s\|f\|_{s}, \quad \forall \lambda \in I.
$$
(4) For any $f, g \in \mathcal B(\Omega)$, $(0,\infty)\setminus
\mathcal E(H) \ni  \lambda \to (R(\lambda \pm i0)f,g)$ is continuous, and if $f \in L^{2,s}(\Omega)$ for $s > 1/2$, $(0,\infty)\setminus \mathcal E(H)\in \lambda \to R(\lambda \pm i0)f \in L^{2,-s}(\Omega)$ is strongly continuous.
\end{theorem}
 The above theorem is proven in the same way as Theorem 3.10 in \cite{IKL10}. Here, we need to mention the radiation
 condition. It suffices to consider the case of model band.
 Consider a solution $u\in \mathcal B^{\ast}$ to the equation
 $(- \Delta_y - (\partial_{n+1})^2 - \lambda)u = f \in \mathcal B$, $\lambda > 0$.
 Let $P_m$ be the eigenprojection for the $m$-th eigenvalue $\lambda_m$ of $- (\partial_{n+1})^2$.
 Then, we have $(- \Delta_y - (\lambda - \lambda_m))P_mu= P_mf$. If $\lambda < \lambda_m$, $P_mu$ will belong to $L^2$.  
 Therefore, the radiation condition is required only when $\lambda > \lambda_m$.
 We say that $u$ satisfies the outgoing radiation condition if, for any $1 \leq j \leq N$
 and $\lambda_{j,m} < \lambda$, $u_{j,m} = P_{j,m}u$ satisfies 
\begin{equation}
\big(\frac{\partial}{\partial |y|} - i\sqrt{\lambda - \lambda_{j,m}}\big)u_{j,m} \simeq 0,
\label{RadCond1}
\end{equation}
 in the sense of (\ref{ExpansioninBast(Omega)}).
 If $i\sqrt{\lambda - \lambda_{j,m}}$ is replaced by $-i\sqrt{\lambda - \lambda_{j,m}}$, $u$ is said to satisfy
 the incoming radiation condition. 

 The condition (\ref{RadCond1}) is rephrased as follows. 
\begin{lemma}
 The condition (\ref{RadCond1}) is equivalent to 
\begin{equation}
 \big(\frac{\partial}{\partial|y|} - i\sqrt{\lambda + (\partial_{j,n+1})^2}\big)u \simeq 0.
\label{RadCond2}
\end{equation}
\end{lemma}
\begin{proof}
 Projecting by $P_{j,m}$, we can derive (\ref{RadCond1}) from (\ref{RadCond2}). We prove the converse.
 Take $\tilde\chi_j\in C^{\infty}(\Omega)$ such that
 $\tilde\chi_j = 1$ on $\Omega_j\cap\{|y|> R+1\}$, $\tilde\chi_j = 0$
 on $\Omega_j\cap\{|y| < R\}$ and on $\Omega\setminus\Omega_j$, $R$ being a sufficiently large constant.
 Letting $u_j = \tilde\chi_j u$, we have
$$
(- \Delta_y - (\partial_{j,n+1})^2 - \lambda)u_j = f_j,
$$
 where $f_j = 0$ on $\Omega_j \cap \{|y| < R\}$ and $\Omega\setminus \Omega_j$.
 Then 
$$
\partial_{n+1}u_j = \sum_{i=1}^n\alpha_i\partial_i u_j\quad 
{\it on}\quad \partial\Omega_j,
$$
 where $\alpha_i = O(|y|^{-\delta_0})$. Letting $\mathcal E_j^{(0)}$ be the operator of extension in Lemma \ref{ExtensionLemma} for the unperturbed case $- \Delta_y - (\partial_{j,n+1})^2$, 
 we put
$$
 v_j = u_j - \mathcal E_j^{(0)}\partial_{n+1}u_j.
$$
 Then $v_j$ satisfies
$$
 (- \Delta_y - (\partial_{j,n+1})^2 - \lambda)v_j = g_j \in L^{2,s}, \quad s > 1/2
$$
 and $\partial_{\nu^{(0)}}v_j = 0$ on $\partial\Omega_j$. 
Note that $\partial_{\nu^{(0)}} = \partial_{n+1}$.  Let $P_j(\lambda) = \sum_{\lambda_{j,m} > \lambda}P_{j,m}$,
 and put $w_j = P_j(\lambda)v_j$. Then $w_j$ satisfies the same equation with $g_j$ replaced by $P_j(\lambda)g_j$.
 As $\lambda_{j,m} > \lambda$, we can invert this equation to have
$$
 w_j = G_jP_j(\lambda)g_j,
$$
 where $G_j \in {\bf B}(H^0(\Omega;H^2(\Omega))$. Then, we have $w_j \in H^2(\Omega_j)$ and satisfies (\ref{RadCond2}).
 Therefore, $u_j$ has the same property, which proves the lemma.
\end{proof}
 The following lemma is then proven in the same way as Lemma 3.8 of \cite{IKL10}.
\begin{theorem}
\label{ExistnceuniquenessHelmholtz}
 Assume $\lambda \in (0, \infty) \setminus \mathcal E(H)$. Then, the solution  $u \in \mathcal B^{\ast}(\Omega)$ to
 the Helmholtz equation $(H - \lambda)u = f \in \mathcal B(\Omega)$ satisfying the outgoing or incoming radiation condition
 is unique.
\end{theorem}
 In the course of the proof of Theorem \ref{LAPforH}, it is shown that the resolvent $R(\lambda \pm i0)$ satisfies the radiation condition. We then have the following corollary.
\begin{cor}
For $\lambda \in (0,\infty)\setminus{\mathcal E}(H)$ and $f \in \mathcal B$, $R(\lambda \pm i0)f$ is a unique solution to the equation $(H - \lambda)u = f$ satisfying the outgoing or incoming radiation condition.
\end{cor}

\subsection{Free spectral representation}
 To derive the eigenfunction expansion, we first consider $H^{(0)}_0 = -\Delta_y -(\partial_{n+1})^2$
 in the euclidean cylinder ${\bf R}^n\times (0,1)$ with Neumann condition on $x_{n+1} = 0, 1$.
 Let  $G_0(z) = (-\Delta_y -z)^{-1}$ be the Green operator for $- \Delta_y$ on ${\bf R}^n$.
 Let $\lambda_{\ell}$ and $P_{\ell}$, $\ell = 1, 2, \dots$ be the eigenvalues and  eigenprojections of $-(\partial_{n+1})^{2}$ on $\lbrack  0, 1\rbrack$. 
 Then for $z \in {\bf C}\setminus [0,\infty)$, $R^{(0)}_0(z) = (H^{(0)}_0 - z)^{-1}$ is written as
\begin{equation}
R^{(0)}_0(z) = \sum _{\ell=1}^{\infty} G_0(z -\lambda_{\ell})\otimes P_{\ell}.
\label{R00=G0Pell}
\end{equation}
 We put
\begin{equation}
c_0(\lambda) =  \frac{\lambda^{(n-2)/4}}{{\sqrt 2}}, \quad \lambda>0,
\label{DefineC0laambda}
\end{equation}
 and define the Fourier transform $\mathcal{F}_0(\lambda)$ on ${\bf R}^n$ by 
\begin{equation}
\mathcal{F}_0(\lambda) f = c_0(\lambda)\hat{f}(\sqrt{\lambda}\omega) =  c_0(\lambda)(2\pi)^{-n/2}\int_{{\bf R}^n}e^{- i\sqrt{\lambda}\omega\cdot y}f(y) dy, \quad \omega \in S^{n-1},
\label{DefineF0lambdaf}
\end{equation}
\begin{equation}
\hat f(\sqrt{\lambda}\omega) = (2\pi)^{-n/2}\int_{{\bf R}^n} e^{-i\sqrt{\lambda}\omega\cdot y}f(y)dy.
\nonumber
\end{equation}
 We also define
\begin{equation}
\mathcal F_0^{(\pm)}(\lambda)f(\omega) = \big(\mathcal F_0(\lambda)f\big)(\pm \omega).
\nonumber
\end{equation}
 It is well-known that for any $f \in \mathcal B({\bf R}^n)$ and $\lambda > 0$,  $G_0(\lambda \pm i0)f$ has the expansion
\begin{equation}
G_0(\lambda \pm i0)f(\omega) \simeq c_1^{(\pm)}(\lambda)r^{-(n-1)/2}
e^{\pm i\sqrt{\lambda}r}\hat f(\pm \sqrt{\lambda}\omega),
\nonumber
\end{equation}
\begin{equation}
c_1^{(\pm)}(\lambda) = \sqrt{\frac{\pi}{2}}e^{\pm (3-n)\pi i/4}\lambda^{(n-3)/4},
\nonumber
\end{equation}
in $\mathcal B^{\ast}({\bf R}^n)$ 
as $r= |y| \to \infty$, $\omega = y/r$.
This together with (\ref{R00=G0Pell})  implies the following lemma.

\begin{lemma}
\label{Lemmaresolventexpandfree}
For $\lambda \in (0,\infty)\setminus\{\lambda_{\ell}\}_{\ell=1}^{\infty}$, the limit
$R^{(0)}_0(\lambda \pm i0) : \mathcal B \to \mathcal B^{\ast}$ exists, and for $f \in \mathcal B$, the expansion
\begin{equation} 
 R^{(0)}_0(\lambda\pm i 0)f\simeq {\mathop\sum _{\lambda_{\ell} <\lambda}}
 c_{\pm}(\lambda - \lambda_{\ell}) \frac{e^{\pm i|y| \sqrt{\lambda -\lambda_{\ell}}}}{\vert y\vert^{(n-1)/2}}\mathcal{F}_0^{(\pm)}(\lambda - \lambda_{\ell})\otimes P_{\ell} f, 
\label{freeR00fexpansion}
\end{equation}
\begin{equation}
 c_{\pm}(\lambda) := e^{\pm (3-n)\pi i/4}\sqrt{\pi}\lambda^{-1/4}, \quad \lambda>0,
\label{Definecpmlambda}
\end{equation}
 holds in the sense of (\ref{ExpansioninBast(Omega)}). 
\end{lemma}

Given a slab domain having $N$ bands, let $\lambda_{j,\ell}$ be the Neumann eigenvalue and
 $\varphi_{j,\ell}$ the normalized Neumann eigenfunction for $- (\partial_{n+1})^2$ of the $j$-th model band
 $\Omega_j = {\bf R}^n\times (0,d_j)$. 
In the following, for the sake of simplicity of description, we assume that $d_j = 1$. 
 We put
\begin{equation}
\Psi_{j,\ell}^{(0)}(y,x_{n+1};\lambda,\omega) = 
c_0(\lambda - \lambda_{j,\ell})(2\pi)^{-n/2}
e^{i\sqrt{\lambda - \lambda_{j,\ell}}\omega\cdot y}
\varphi_{j,\ell}(x_{n+1}), \quad \lambda > \lambda_{j,\ell}.
\label{DefinePsijell0ycn+1}
\end{equation}
 In view of (\ref{DefineF0lambdaf}) and (\ref{freeR00fexpansion}), we define the free Fourier transformation
\begin{equation}
\mathcal F_{j,\ell}^{(0)}(\lambda)f(\omega,x_{n+1}) =
\left(\int_{{\bf R}^n\times (0,1)}
\overline{\Psi^{(0)}_{j,\ell}(y,x_{n+1}';\lambda,\omega)}f(y,x_{n+1}')dydx_{n+1}'\right) \varphi_{j,\ell}(x_{n+1}),
\nonumber
\end{equation}
\begin{equation}
\mathcal F^{(0)}_j f(\lambda) = \sum_{\ell=1}^{\infty}
\mathcal F^{(0)}_{j,\ell}(\lambda)f
\nonumber
\end{equation}
 for each model band $\Omega_j$, which is first defined on $C_0^{\infty}(\Omega_j)$
 and then extended to whole $L^2(\Omega_j)$. 
 The total free Fourier transformation is defined by
\begin{equation}
\mathcal{F}^{(0)}= (\mathcal{F}^{(0)}_1, \mathcal{F}^{(0)}_2, \ldots, \mathcal{F}^{(0)}_N), 
\label{totalFouriertransformF0}
\end{equation}
 which is unitary from $L^2(\cup_{j=1}^N\Omega_j)$ to $\oplus_{j=1}^NL^2((0,\infty);L^2(S^{n-1}\times
  (0,d_j));d\lambda)$ 
 \footnote{Here, for an interval $I \subset {\bf R}$ and a Hilbert space $\mathcal X$, $L^2(I;\mathcal X;d\lambda)$
 denotes the set of all $\mathcal X$-valued $L^2$-functions on $I$ with respect to the measure $d\lambda$.}.

\subsection{Perturbed spectral representation}
\label{Perturbedspectralrepresent}
When $\delta_0 > (n+1)/2$,
the generalized eigenfunctions for $H$ are defined by
\begin{equation}
 \Psi_{j,\ell,\pm}(\lambda) = (\chi_j -\mathcal{E}_jB_j)\Psi^{(0)}_{j,\ell}(\lambda)
 - R(\lambda  \textcolor{blue}{\pm} i 0)\mathcal V_j(\lambda) \Psi^{(0)}_{j,\ell}(\lambda).
\label{Psijlpm=...-R(lambda-i0)}
\end{equation}
 It satisfies
\begin{equation}
\left\lbrace
\begin{split}
&(H-\lambda)\Psi_{j,\ell,\pm}(\lambda) = 0\quad  \mbox{in}\quad \Omega,\\
&\partial _\nu \Psi_{j,\ell,\pm}(\lambda) = 0 \quad \mbox{on} \quad \partial\Omega.
\end{split}
\right.
\nonumber
\end{equation}

 Let ${\bf h}_j(\lambda) = L^2(S^{n-1})\times {\rm span}\,\{\varphi_{j,\ell} \, ; \, \lambda_{j,\ell} < \lambda\}$, where ${\rm span}\, A $ means the linear hull of the set $A$, 
 and put
\begin{equation}
{\bf h}(\lambda) = \oplus_{j=1}^N{\bf h}_j(\lambda).
\label{hlambdadefine}
\end{equation}

 To deal with the general case, we consider the operator (cf. (\ref{chijR(z)=tildechijJj-1R(0)j(z)}))
\begin{equation}
\mathcal{F}_{j,\ell,\pm}(\lambda) = \mathcal{F}^{(0)}_{j,\ell}\big(\lambda) J_j \Big(\chi_j -(\mathcal{E}_jB_j)^{\ast} -\mathcal V_j(\lambda)^{\ast}R(\lambda \mp i 0)\Big),
\label{mathcalFjellpmlambdadefine}
\end{equation}
 which is well-defined on $\mathcal B$ when $\delta_0 > 1$, i.e. 
\begin{equation}
 \mathcal F_{j,\ell,\pm}(\lambda) \in {\bf B}(\mathcal B; {\bf h}_j(\lambda)), \quad \lambda \in (0,\infty)\setminus \mathcal E(H).
 \nonumber
\end{equation}
If $\delta_0 > (n+1)/2$,  by (\ref{Psijlpm=...-R(lambda-i0)}), $\mathcal{F}_{j,\ell,\pm}(\lambda)^{\ast}$ is the operator with integral kernel $\Psi_{j,\ell,\pm}(y, x_{n+1};\lambda,\omega)$.
 Letting $\chi_{\lambda_{j,\ell}}(\lambda)$ be the characteristic function of the interval $(\lambda_{j,\ell},\infty)$, 
 we define
\begin{equation}
\mathcal{F}_{j,\pm}(\lambda) = \sum_{\ell=1}^{\infty} \chi_{\lambda_j,\ell}(\lambda) \mathcal{F}_{j,\ell,\pm}(\lambda) = \sum _{\lambda_{j,\ell}<\lambda}\mathcal{F}_{j,\ell,\pm}(\lambda),
\nonumber
\end{equation}
\begin{equation}
\mathcal{F}_\pm (\lambda) = (\mathcal{F}_{1,\pm}(\lambda), \ldots, \mathcal{F}_{N,\pm}(\lambda)).
\nonumber
\end{equation}
 By the resolvent equations (\ref{chijR(z)=tildechijJj-1R(0)j(z)1}), (\ref{chijR(z)=tildechijJj-1R(0)j(z)}),
 the expansion in Lemma \ref{Lemmaresolventexpandfree} is  extended to the perturbed operator.
\begin{lemma}
\label{LemmaRlambdafexpansion}
 For any $\lambda\in (0,\infty)\setminus \mathcal{E}(H)$ and $f\in \mathcal{B},$ we have  on $\Omega_j$
\begin{equation} 
 R(\lambda\pm \imath 0)f\simeq {\mathop\sum_{\lambda_{j,\ell} <\lambda}}c_{\pm}(\lambda - \lambda_{j,\ell}) \frac{e^{\pm i  \sqrt{\lambda -\lambda_{j,\ell}}\vert y\vert}}{\vert y\vert^{(n-1)/2}}\mathcal{F}_{j,\ell,\pm}(\lambda) f.
\label{Rlambdapmi0fexpansion1}
\end{equation}
\end{lemma}

 The following Parseval's formula is proven in the same way as \cite{IKL10}, Lemma 4.5.
\begin{lemma}
 For any $\lambda \in (0,\infty) \setminus\mathcal{E}(H)$ and $f\in \mathcal{B}$, we have
\begin{equation}
\frac{1}{2\pi \imath}\big((R(\lambda + i 0) -R(\lambda -i 0))f,f\big) =
 \Vert \mathcal{F}_{\pm}(\lambda)f\Vert^2_{{\bf h}(\lambda)}
 \nonumber
\end{equation}
\end{lemma}

We define
\begin{equation}
 \widehat {\mathbb H}_j = \left\{\sum_{\ell=1}^{\infty} f_{\ell}(\xi)\varphi_{j,\ell}(x_{n+1})\, ;
  \, f_{\ell} \in L^2(|\xi|^2 > \lambda_{j,\ell}),\;
  \sum_{\ell=1}^{\infty} \|f_{\ell}\|^2_{L^2(|\xi|^2 > \lambda_{j,\ell})}<\infty \right\},
  \nonumber
\end{equation}
 where $ L^2(|\xi|^2 > \lambda_{j,\ell})$ is the space of functions $f \in L^2({\bf R}^n)$ with support
 in $|\xi|^2 > \lambda_{j,\ell}$. We put
\begin{equation}
 \widehat{\mathbb H} = \oplus_{j=1}^N\widehat{\mathbb H}_j.
\end{equation}
 These preparations are sufficient to prove the following theorem. Let ${\mathcal H}_{ac}(H)$ be the absolutely continuous subspace for $H$.
\begin{theorem}
(1) For $\lambda \in (0,\infty)\setminus \mathcal{E}(H)$, $\mathcal{F}_{\pm}(\lambda) \in
 {\bf B}(\mathcal{B};{\bf h}(\lambda)).$

\noindent
(2) The operator $(\mathcal{F}_\pm f)(\lambda) = \mathcal{F}_\pm(\lambda)f$ defined for $f\in \mathcal{B}$ is uniquely extended to a partial isometry with initial set $\mathcal H_{ac}(H)$ and final set $\widehat{\mathbb H}$, and 
$(\mathcal{F}_{\pm}Hf)(\lambda) = \lambda(\mathcal{F}_\pm f)(\lambda)$ for $ \lambda\in (0,\infty)\setminus \mathcal{E}(H), \; f\in D(H)$.

\noindent
(3) $\mathcal{F}_\pm(\lambda)^\ast \in {\bf B}({\bf h}(\lambda) ;\mathcal{B}^\ast)$ is an eigenoperator of $H$ with eigenvalue $\lambda$ in the sense that
\begin{equation}
(H-\lambda)\mathcal{F}_\pm (\lambda)^\ast\psi = 0, \; \forall \psi \in  {\bf h}(\lambda).
\label{EigenEquation}
\end{equation}
(4) For any compact interval $I\subset (0,\infty)\setminus \mathcal{E}(H)$ and $g\in \hat{\mathbb H}$, we have
\begin{equation}
\int_I \mathcal{F}_\pm(\lambda)^\ast g(\lambda)d\lambda \in L^2(\Omega).
\nonumber
\end{equation}
Let $I_n$ be a finite union of compact intervals in $(0,\infty)\setminus \mathcal{E}(H)$ such that $I_n\subset I_{n+1}, \cup_{n=1}^\infty I_n = (0,\infty)\setminus \mathcal{E}(H).$ Then for any $f\in \mathcal{H}_{ac }(H)$, the inversion formula holds in $L^2(\Omega)$:
$$
 f =
 {\mathop{\rm{s-lim}}_{n\to \infty}} \int_{I_n}\mathcal{F}_\pm(\lambda)^\ast (\mathcal{F}_\pm f )(\lambda) d\lambda,
$$
 where the limit is taken in $L^2(\Omega)$.
\end{theorem}

\subsection{S-matrix} 
 The time-dependent scattering theory can be developed in the same way as in Theorem 4.7 of \cite{IKL10}. 
 Let 
\begin{equation}
 H_j^{(0)} = - \Delta_y - (\partial_{j,n+1})^2
 \nonumber
\end{equation}
 be the Laplacian on the slab $\Omega_j$. The wave operator $W_{\pm} : \oplus_{j=1}^NL^2(\Omega_j) \to L^2(\Omega)$
 is defined by
\begin{equation}
 W_{\pm} = {\mathop{\rm s-lim}_{t\to \pm\infty}}\sum_{j=1}^N e^{it\sqrt{H}}\chi_j e^{-it\sqrt{H_j^{(0)}}}.
 \nonumber
\end{equation}
 Then, $W_{\pm}$ exists and is complete, i.e. ${\rm Ran}\, W_{\pm} = \mathcal H_{ac}(H)$.
 Moreover, $W_{\pm} = (\mathcal F_{\pm})^{\ast}\mathcal F^{(0)}$ (see (\ref{totalFouriertransformF0})).  The scattering operator is defined by
\begin{equation}
 S = (W_+)^{\ast}W_-.
 \nonumber
\end{equation}
 We consider its Fourier transform:
\begin{equation}
 \widehat S = \mathcal F^{(0)}S(\mathcal F^{(0)})^{\ast}.
 \nonumber
\end{equation}
 It has the following representation. For each $\lambda \in (0,\infty)\setminus{\mathcal E}(H)$, there exists
 a unitary operator $\widehat S(\lambda)$ on ${\bf h}(\lambda)$ such that
\begin{equation}
 (\widehat Sf)(\lambda) = \widehat S(\lambda)f(\lambda), \quad \forall f \in \widehat H,
  \quad \lambda \in (0,\infty) \setminus \mathcal E(H).
  \nonumber
\end{equation}
 Here $\widehat S(\lambda)$ is an $N\times N$ matrix operator whose entry has the form (\cite{IKL10}, Lemma 4.9)
\begin{equation}
\widehat S_{jk}(\lambda) = \delta_{jk}I_k - 2\pi i \mathcal F_{j,+}(\lambda)
  \mathcal V_k(\lambda)\big(\mathcal F_k^{(0)}(\lambda)\big)^{\ast}, \quad 1 \leq j, k \leq N.
  \nonumber
\end{equation}
$I_k$ being the identity on ${\bf h}_k(\lambda)$. 
 The scattering amplitude is defined by
\begin{equation}
 A_{jk}(\lambda) = \mathcal F_{j,+}(\lambda)\mathcal V_k(\lambda)\big(\mathcal F_k^{(0)}(\lambda)\big)^{\ast}
 \in {\bf B}({\bf h}_k(\lambda) \, ; \,{\bf h}_j(\lambda)).
\label{DefineAjklambda}
\end{equation}
 Projecting to the $m$-th and the $n$-th eigenspaces for $- (\partial_{j,n+1})^2$ and $- (\partial_{k,n+1})^2$,  we put
\begin{equation}
A_{jm,kn}(\lambda) = \mathcal F_{j,m,+}(\lambda)\mathcal V_k(\lambda)\big(\mathcal F_{k,n}^{(0)}(\lambda)\big)^{\ast}.
\nonumber
\end{equation}
Then, we have
\begin{equation}
\label{rel.SjkAjmkn}
\widehat S_{jk}(\lambda) - \delta_{jk}I_k = - 2\pi i \sum_{\lambda_{j,m}<\lambda, \:
 \lambda_{k,n}<\lambda} A_{jm,kn}(\lambda).
\end{equation}

We now fix $j$ and $k$,  and assume that $\delta_0 > (n+1)/2$ on the slabs $\Omega_j$ and $\Omega_k$. Then, the operator 
 $A_{jm,kn}(\lambda)$ has an integral kernel written by $\Psi_{j,m,+}(\lambda)$ and
 $\Psi^{(0)}_{k,n}(\lambda)$\footnote{With a slight abuse of notation in the arguments of $\Psi_{j,m,+}(\lambda)$ and $\Psi^{(0)}_{k,n}(\lambda)$.}:
\begin{equation}
A_{jm,kn}(\lambda;\omega,x_{n+1},\omega',x_{n+1}') = \int_{\Omega}\overline{\Psi_{j,m,+}(\lambda; X,\omega,x_{n+1})}\Big(\mathcal V_{k}(\lambda)\Psi^{(0)}_{k,n}(\lambda;X,\omega',x_{n+1}')\Big)d\Omega_X.
\nonumber
\end{equation}
We contract the variables $x_{n+1}, x'_{n+1}$, and define $\mathcal A_{jm,kn}(\lambda;\omega,\omega')$ by 
\begin{equation}
\mathcal A_{jm,kn}(\lambda;\omega,\omega') = 
\iint_{(0,1)\times (0,1)}\varphi_{j,m}(x_{n+1})A_{jm,kn}(\lambda;\omega,x_{n+1},\omega',x_{n+1}')\varphi_{k,n}(x_{n+1}')dx_{n+1}dx'_{n+1}.
\label{DefinemathcaAjmkncontraction}
\end{equation}
Let $\mathcal A_{jm,kn}(\lambda)$ be the integral operator with kernel $\mathcal A_{jm,kn}(\lambda;\omega,\omega')$.
Then letting
\begin{equation}
\langle h,\varphi_{k,n}\rangle = \int_0^1h(\omega,x_{n+1})\varphi(x_{n+1})dx_{n+1},
\nonumber
\end{equation}
we have for $h \in {\bf h}_{k}(\lambda)$
\begin{equation}
A_{jm,kn}(\lambda)h = \varphi_{j,m}\mathcal A_{jm,kn}(\lambda)\langle h,\varphi_{k,n}\rangle.
\nonumber
\end{equation}
By virtue of (\ref{Psijlpm=...-R(lambda-i0)}) and Lemma \ref{LemmaRlambdafexpansion}, we have
\begin{equation}
\Psi_{j,\ell,\pm}(\lambda) - \chi_j\Psi_{j,\ell}^{(0)}(\lambda) \simeq 
- \sum_{\lambda_{j,m}<\lambda}c_{\mp}(\lambda - \lambda_{j,m})
\frac{e^{\mp i\sqrt{\lambda - \lambda_{j,m}}|y|}}{|y|^{(n-1)/2}}
\mathcal F_{j,m,\mp}(\lambda)\mathcal V_j(\lambda)\Psi^{(0)}_{j,\ell}(\lambda).
\nonumber
\end{equation}
This formula implies the following lemma.
\begin{lemma}
\label{LemmaPjmPsi-Psi0asympto}
Letting $P_{j,m}$ be the eigenprojection for the $m$-th eigenvalue of $- (\partial_{j,n+1})^2$ on the $j$-th band $\Omega_j$, we have
\begin{equation}
 P_{j,m}\big(\Psi_{j,\ell,-}(\lambda) - \chi_j\Psi^{(0)}_{j,\ell}(\lambda)\big)
 \simeq - c_{+}(\lambda - \lambda_{j,m})
 \frac{e^{i\sqrt{\lambda - \lambda_{j,m}}|y|}}{|y|^{(n-1)/2}}\mathcal A_{jm,kn}(\lambda)\varphi_{j,m}.
 \nonumber
\end{equation}
\end{lemma}

\subsection{Analytic continuation of the scattering amplitude}
We now observe the slab $\Omega_1$. When $j = k = 1$, the scattering amplitude is written as (see  (\ref{DefineF0lambdaf}), (\ref{DefinePsijell0ycn+1}),  (\ref{mathcalFjellpmlambdadefine}) and (\ref{DefineAjklambda}))\footnote{In the following, we often omit the symbol $\otimes$.}
\begin{equation}
\begin{split}
 A_{1m,1n}(\lambda) = \mathcal F_0(\lambda - \lambda_{1,m})P_{1,m}J_1\Big(\chi_1 - (\mathcal E_1B_1)^{\ast}
 - \mathcal V_1(\lambda)^{\ast}R(\lambda + i0) \Big)\mathcal V_1(\lambda)\mathcal F_0(\lambda - \lambda_{1,n})^{\ast}
 P_{1,n}.
\end{split}
\nonumber
\end{equation}
Assuming that $\Omega_1$ is flat for $|y| > R$, we take $\chi_1(y)$ so that $\chi_1(y) = 0$ for $|y| < R +1$, and $\chi_1(y) = 1$ for $|y| > R+2$. Then, in view of (\ref{Bjdefine}), we have $B_1 = 0$ and that
$\mathcal V_1(z)$ in (\ref{DefineVj(z)}) is independent of $z$ and supported in $|y| < R$. 
Then $A_{1m,1n}(\lambda)$ is written as 
\begin{equation}
\begin{split}
 A_{1m,1n}(\lambda) = \mathcal F_0(\lambda - \lambda_{1,m})P_{1,m}
\Big(\chi_1  
 - \mathcal V_1^{\ast}R(\lambda + i0) \Big)\mathcal V_1\mathcal F_0(\lambda - \lambda_{1,n})^{\ast}
 P_{1,n}.
\end{split}
\label{NewA1m1n}
\end{equation}
Then, $A_{1m,1n}(\lambda)$ defined for $\lambda > \max\{\lambda_{1,m},\lambda_{1,n}\}$ is analytically continued to the upper-half plane ${\bf C}_+ = \{{\rm Im}\, \lambda > 0\}$, and is extended to a continuous function on 
${\bf C}_+\cup ({\bf R}\setminus \mathcal E(H))$. We denote the obtained function for $\lambda < \max\{\lambda_{1,m},\lambda_{1,n}\}$ by $A_{1m,1n}^{(nph)}(\lambda)$ and call it the {\it non-physical scattering amplitude}. 
It then follows from this definition that the non-physical scattering amplitude $A_{1m,1n}^{(nph)}(\lambda)$ coincides with the physical scattering amplitude 
$A_{1m,1n}(\lambda)$ for $\lambda > \max\{\lambda_{1,m},\lambda_{1,n}\}$. 
 Since  the function $\sqrt{\lambda - \lambda_{\ell}}$ defined for $\lambda > \lambda_{\ell}$  
 is analytically extended to $i\sqrt{\lambda_{\ell} - \lambda}$ defined for $\lambda < \lambda_{\ell}$, the following
 Lemma \ref{LemmaA1m1nanalyticcont} is obvious from (\ref{NewA1m1n}). We put for $\mu < 0$
\begin{equation}
\mathcal G_{0}(\mu)f = c_0( - \mu)(2\pi)^{-n/2}\int_{{\bf R}^n}e^{\sqrt{-\mu}\omega\cdot y}f(y)dy,
\label{DefinemathcalG0mu}
\end{equation} 
\begin{equation}
\widetilde{\mathcal G_{0}(\mu)} \psi = c_0(-\mu)(2\pi)^{-n/2}\int_{S^{n-1}}e^{- \sqrt{-\mu}\omega\cdot y}\psi(\omega)d\omega.
\nonumber
\end{equation}

\begin{lemma}
\label{LemmaA1m1nanalyticcont}
(1) If $\lambda_{1,m} < \lambda < \lambda_{1,n}$, 
\begin{equation}
\begin{split}
 A_{1m,1n}^{(nph)}(\lambda) = \mathcal F_0(\lambda - \lambda_{1,m})P_{1,m}\Big(\chi_1
  - \mathcal V_1^{\ast}R(\lambda + i0) \Big) \mathcal V_1 \widetilde{\mathcal G_{0}}(\lambda - \lambda_{1,n})
 P_{1,n}.
\end{split}
\end{equation}
(2) If $\lambda_{1,n} < \lambda < \lambda_{1,m}$, 
\begin{equation}
 A_{1m,1n}^{(nph)}(\lambda) = \mathcal G_{0}(\lambda - \lambda_{1,m})P_{1,m}\Big(\chi_1
 - \mathcal V_1^{\ast}R(\lambda + i0) \Big)
\mathcal V_1
 \mathcal F_0(\lambda - \lambda_{1,n})^{\ast}P_{1,n}.
\end{equation}
(3) If $\lambda < \min\{\lambda_{1,m}, \lambda_{1,n}\}$, 
\begin{equation}
 A_{1m,1n}^{(nph)}(\lambda) = \mathcal G_{0}(\lambda-\lambda_{1,m})P_{1,m}\Big(\chi_1 - \mathcal V_1^{\ast}R(\lambda + i0) \Big) \mathcal V_1
 \widetilde{\mathcal G_{0}}(\lambda - \lambda_{1,n}) P_{1,n}.
\end{equation}
\end{lemma}

For $\lambda < \lambda_{1,n}$, we put
\begin{equation}
\Phi^{(0)}_{1,n}(x,\lambda,\omega) = c_0(\lambda_{1,n} - \lambda))
e^{-y\cdot\omega\sqrt{\lambda_{1,n}-\lambda}}\varphi_{1,n}(x_{n+1}),
\label{FormulaPhi01n=e-yomegavarphi1}
\end{equation}
which is an exponentially growing solution to the equation $(- \Delta_y - (\partial_{1,n+1})^2)u = \lambda u$ and
\begin{equation}
\widetilde{\mathcal G}_0(\lambda - \lambda_{1,n})\psi \otimes \varphi_{1,n} = (2\pi)^{-n/2}\int_{S^{n-1}}\Phi^{(0)}_{1,n}(x,\lambda,\omega)\psi(\omega)d\omega, \quad 
\psi \in L^2(S^{n-1}).
\nonumber
\end{equation}
Define the non-physical eigenfunction by
\begin{equation}
\Phi_{1,n,-}(\lambda) = \chi_1 \Phi_{1,n}^{(0)}(\lambda) - R(\lambda + i0)\mathcal V_1\Phi^{(0)}_{1,n}(\lambda),
\label{FormulaPhi1n-=byPhi01n}
\end{equation}
which is an exponentially growing solution to the equation $(H - \lambda)u = 0$. Compared with the physical eigenfunction defined  by (\ref{Psijlpm=...-R(lambda-i0)}) for $\lambda > \lambda_{1,n}$
\begin{equation}
\Psi_{1,\ell,-}(\lambda)  = \chi_1\Psi^{(0)}_{1,\ell}(\lambda) - R(\lambda + i0)\mathcal V_1\Psi^{(0)}_{1,\ell}(\lambda),
\nonumber
\end{equation}
the non-physical eigenfunction is its analytic continuation to the interval $\lambda < \lambda_{1,n}$. By Lemma \ref{LemmaA1m1nanalyticcont}, the non-physical scattering amplitude is computed from the asymptotic behavior of non-physical eigenfunction in the following way. First let us note the following lemma.

\begin{lemma}
\label{ExpansionofDelta-mu-1f}
For compactly supported $f \in L^2({\bf R}^n)$ and $\mu < 0$, we have as $r \to \infty$
\begin{equation}
(- \Delta - \mu)^{-1}f \sim \Big(\frac{\pi}{\sqrt{- \mu}}\Big)^{1/2}\frac{e^{-\sqrt{-\,\mu}r}}{r^{(n-1)/2}}\mathcal G_0(\mu)f.
\nonumber
\end{equation}
\end{lemma}
\begin{proof}
The fundamental solution of $- \Delta - z$ in ${\bf R}^n$ is written as
\begin{equation}
\frac{i}{4}\Big(\frac{\sqrt{z}}{2\pi|x-y|}\Big)^{(n-2)/2}H^{(1)}_{(n-2)/2}(\sqrt{z}|x-y|),
\nonumber
\end{equation}
where $H^{(1)}_{(n-2)/2}(t)$ is the Hankel function of the 1st kind, which has the folloing asymptotic expansion
\begin{equation}
H^{(1)}_{\nu}(t) \sim \sqrt{\frac{2}{\pi t}}e^{i(t - (2\nu + 1)\pi/4)}, \quad |t| \to \infty, \quad \textcolor{blue}{ - \pi} < {\rm arg}\, t < \textcolor{blue}{2\pi}.
\nonumber
\end{equation}
(See \cite{MOS}, p.139.)
Then, if $f \in L^2({\bf R}^n)$ is compactly supported, we have for $\mu < 0$
\begin{equation}
\begin{split}
( - \Delta - \mu)^{-1}f &\sim \frac{1}{4\pi}\Big(\frac{\sqrt{-\mu}}{2\pi}\Big)^{(n-3)/2}
\int_{{\bf R}^n}e^{- \sqrt{-\mu}|x-y|}|x-y|^{-(n-1)/2}f(y)dy \\ &\sim
\frac{1}{4\pi}\Big(\frac{\sqrt{-\mu}}{2\pi}\Big)^{(n-3)/2}\frac{e^{-\sqrt{-\mu}r}}{r^{(n-1)/2}}\int_{{\bf R}^n}e^{\sqrt{-\mu}\omega\cdot y}f(y)dy
\end{split}
\nonumber
\end{equation}
as $r= |x| \to \infty$. Using (\ref{DefineC0laambda}) and (\ref{DefinemathcalG0mu}), we obtain the lemma.
\end{proof}

We put as in (\ref{DefinemathcaAjmkncontraction})
$$
\mathcal A^{(nph)}_{1m,1n}(\lambda) = (A^{(nph)}_{1m,1n}(\lambda)\varphi_{1,n},\varphi_{1,m}).
$$
\begin{lemma}
(1) If $\lambda_{1,m} < \lambda < \lambda_{1,n}$, we have as $|y| \to \infty$\footnote{See (\ref{Definecpmlambda}) for $c_+(\lambda)$. Note that here the $n$ in $e^{(2-n)\pi i/4}$ refers to the space dimension $n$ of ${\bf R}^n$.},
$$
P_{1,m}\big(\Phi_{1,n,-}(\lambda) - \Phi_{1,n}^{(0)}(\lambda)\big) \simeq 
- c_+ (\lambda - \lambda_{1,m})\frac{e^{i |y|\sqrt{\lambda - \lambda_{1,m}}}}{|y|^{(n-1)/2}}
\mathcal A_{1m;1n}^{(nph)}(\lambda)\varphi_{1,n}. 
$$
(2) If $\lambda < \min\{\lambda_{1,m}, \lambda_{1,n}\}$, we have as $|y| \to \infty$,
$$
 P_{1,m}\big(\Phi_{1,n,-}(\lambda) - \Phi_{1,n}^{(0)}(\lambda)\big) \sim
 - \Big(\frac{\pi}{\sqrt{\lambda_{1,m}-\lambda}}\Big)^{1/2}\frac{e^{-|y|\sqrt{\lambda_{1,m} - \lambda}}}{|y|^{(n-1)/2}}
 {\mathcal A}_{1m;1n}^{(nph)}(\lambda)\varphi_{1,n}.
$$
\end{lemma}
\begin{proof} First we prove (2). Let us note that by (\ref{chijR(z)=tildechijJj-1R(0)j(z)}),  we have
\begin{equation}
\tilde\chi_1 R(\lambda + i0) = \tilde \chi_1 R_1^{(0)}(\lambda + i0)J_1\Big(\chi_1 - (\mathcal E_1B_1)^{\ast} - \mathcal V_1^{\ast}R(\lambda + i0)\Big).
\label{tildeRlambda0=R0lambda0}
\end{equation}
In fact, as $\Omega_1$  is flat for $|y| > R$, we construct $B_1$ to be 0 for $|y| > R + 1$, and then take $\chi_1(y)$ so that $\chi_1(y) = 0$ for $|y| < R+1$, $\chi_1(y) = 1$ for $|y| > R+2$. Then, as $J_1 = 1$, we obtain (\ref{tildeRlambda0=R0lambda0}). 

In view of (\ref{FormulaPhi01n=e-yomegavarphi1}), (\ref{FormulaPhi1n-=byPhi01n}), we have for large $|y|$
\begin{equation}
\begin{split}
\Phi_{1,n,-}(\lambda) - \Phi^{(0)}_{1,n}(\lambda) &= - R(\lambda + i0)\mathcal V_1\Phi^{(0)}_{1,n}(\lambda) \\
&= - R^{(0)}_{1}(\lambda + i0)
J_1\Big(\chi_1 - (\mathcal E_1B_1)^{\ast} - 
\mathcal V_1^{\ast}R(\lambda + i0)\Big)\mathcal V_1\Phi^{(0)}_{1,n}(\lambda) \\
&= - R^{(0)}_{1}(\lambda + i0)
\Big(\chi_1 - 
\mathcal V_1^{\ast}R(\lambda + i0)\Big)\mathcal V_1\Phi^{(0)}_{1,n}(\lambda). 
\end{split}
\nonumber
\end{equation}
Mutiplying $P_{1,m}$ and using (\ref{R00=G0Pell}), we have
\begin{equation}
\begin{split}
& P_{1,m}\Big(\Phi_{1,n,-}(\lambda) - \Phi^{(0)}_{1,n}(\lambda)\Big) \\
&= - ( - \Delta_y - \lambda + \lambda_{1,m})^{-1}P_{1,m}J_1  
\Big(\chi_1 - 
\mathcal V_1^{\ast}R(\lambda + i0)\Big)\mathcal V_1\Phi^{(0)}_{1,n}
\end{split}
\nonumber
\end{equation}
Applying Lemma \ref{ExpansionofDelta-mu-1f}, we have
\begin{equation}
\begin{split}
& P_{1,m}\Big(\Phi_{1,n,-}(\lambda) - \Phi^{(0)}_{1,n}(\lambda)\Big) \\
& \sim \Big(\frac{\pi}{\sqrt{\lambda_{1,m}-\lambda}}\Big)^{1/2}
\frac{e^{-|y|\sqrt{\lambda_{1,m}-\lambda}}}{|y|^{(n-1)/2}}\mathcal G_0(\lambda - \lambda_{1,m})J_1P_{1,m}\Big(\chi_1 - 
\mathcal V_1^{\ast}R(\lambda + i0)\Big)\mathcal V_1\Phi^{(0)}_{1,n}(\lambda).
\end{split}
\nonumber
\end{equation}
Lemma \ref{LemmaA1m1nanalyticcont} (3) then implies the assertion (2).  The assertion (1) can be proven similarly, using Lemma \ref{Lemmaresolventexpandfree}.
\end{proof}

We put
\begin{equation}
\mathcal F^{(phy)}_{1,-}(\lambda)f = 
\sum_{\lambda > \lambda_{\ell}}\int_{\Omega} \overline{\Psi_{1,\ell,-}(\lambda;X,\omega, x_{n+1})}f(X)d\Omega_X,
\nonumber
\end{equation}
\begin{equation}
\mathcal F^{(nph)}_{1,-}(\lambda)f = 
\sum_{\lambda < \lambda_{\ell}}\int_{\Omega} \overline{\Phi_{1,\ell,-}(\lambda;X, \omega, x_{n+1})}f(X)d\Omega_X.
\nonumber
\end{equation}

 Arguing in the same way as Lemma \ref{LemmaRlambdafexpansion}, we obtain the following lemma.
\begin{lemma} 
\label{P1lRlambda+i0fasympto}
 Let $f \in C_0^{\infty}(\Omega)$. \\
\noindent
(1) If $\lambda > \lambda_{1,\ell}$, we have on $\Omega_1$
$$
 P_{1,\ell}R(\lambda + i0)f \simeq c_+(\lambda - \lambda_{1,\ell})\, \frac{e^{i|y|\sqrt{\lambda - \lambda_{1,\ell}}}}
 {|y|^{(n-1)/2}}P_{1,\ell}\mathcal F^{(phy)}_{1,+}(\lambda)f.
$$
(2) If $\lambda < \lambda_{1, \ell}$, we have on $\Omega_1$
$$
 P_{1,\ell}R(\lambda + i0)f \sim (\lambda_{1,\ell} - \lambda)^{(n-3)/4}\, \frac{e^{-|y|\sqrt{\lambda - \lambda_{1,\ell}}}}
 {|y|^{(n-1)/2}}P_{1,\ell}\mathcal F^{(nph)}_{1,+}(\lambda)f.
$$
\end{lemma}

\section{From the scattering matrix to the 
interior problem}
\label{SectionFromScatteringtoInterior}
In this section, we consider the reconstruction of the domain $\Omega$. Suppose we are given two slab domains $\Omega^{(i)} = \mathcal K^{(i)} \cup \Omega_1^{(i)} \cup \cdots \cup \Omega_{N^{(i)}}^{(i)}$, $i = 1, 2$,  satisfying the assumptions in \S \ref{SectionIntro}. 
Note that we do not assume $N^{(1)} = N^{(2)}$. It will be proven in our inverse procedure given below.

The first step is to reduce the issue to the \textit{interior} boundary value problem. For this purpose, there are two ways. The first one is used in \cite{IKL10}, and the second one is used in \cite{IL26(2)}.

\subsection{Boundary spectral projections}
The first  method uses the boundary spectral projection. 
Take a constant $R_0 > 0$ so that $\Omega_1 \cap \{|y| > R_0\}$ is flat. We put 
$$
\Omega_{ext} = \Omega_1 \cap \{|y| > R_0\}, \quad \Omega_{int} = \Omega \setminus \overline{\Omega_{int}}.
$$
Take an open set $\mathcal O \subset\subset \Omega_{int}$ such that it has a smooth boundary not intersecting $\partial\Omega_{int}$ and that $\Omega_{int}\setminus \mathcal O$ is connected. Let $H_{\mathcal O}$ be $- \Delta_G$ on $\Omega_{int} \setminus \overline{\mathcal O}$ with Neumann boundary condition.
 We can construct the generalized Fourier transformation $\mathcal F(\lambda)$ for $H_{\mathcal O}$ as in Subsection \ref{Perturbedspectralrepresent}.
 Let $(\lambda_i, P_i)$ be the set of eigenvalues and eigenprojections for $H_{\mathcal O}$. We put
 $\Gamma_{\mathcal O} = \Omega_1 \cap \{|y| = R_0\}$ if $\mathcal O = \emptyset$, and $\Gamma_{\mathcal O} = \partial \mathcal O$ if $\mathcal O \neq \emptyset$.
Let $r_{\mathcal O} \in {\bf B}(H^1(\Omega_{\mathcal O});H^{1/2}(\Gamma_{\mathcal O}))$ be the trace operator to $\Gamma_{\mathcal O}$,
\begin{equation}
r_{\mathcal O} : H^1(\Omega_{\mathcal O}) \ni f \to 
f\big|_{\partial{\mathcal O}} \in H^{1/2}(\Gamma_{\mathcal O}).
\nonumber
\end{equation}

 We call the set
\begin{equation}
\left\{(\lambda,r_{\mathcal O}\mathcal F(\lambda)^{\ast}\mathcal F(\lambda)
r_{\mathcal O}^\ast);\lambda \in (0,\infty)\setminus \mathcal E(H_{\mathcal O})\right\} \cup 
\left\{(\lambda_i,r_{\mathcal O} P_i r_{\mathcal O}^{\ast}) \right\}_{i=1}^{\textcolor{blue}{d_p}}
\nonumber
\end{equation}
the boundary spectral projection (BSP) for $H_{\mathcal O}$ on $\Gamma_{\mathcal O}$, where $d_p$ is the dimension of the point spectral subspace for $H_{\mathcal O}$. 
Then, arguing in the same way as in \cite{IKL10} Subsections 5.2 and 5.3, one can prove that $S_{11}(\lambda)$, the $(1,1)$-component of the S-matrix,  determines BSP. See Lemma 5.7 of \cite{IKL10}.

\subsection{Source-to-solution map}
The second method uses the source-to-solution map. 
Using the above notations $\Omega_{ext}$ and $\Omega_{int}$,
 assume that $\Omega_{ext}$ is flat and $H = H^{(0)}_0$ there. 
 Take a bounded open set $\mathcal O \subset\subset \Omega_{ext}$, and consider the following boundary value problem
\begin{equation}
\left\{
\begin{split}
& (H - \lambda)u = F \quad {\it in}\quad \Omega, \quad {\rm supp}\, F \subset \mathcal O,\\
& \partial_{\nu}u = 0 \quad {\it on} \quad \partial\Omega, \\
& u \ {\it satisfies}\  {\it the} \ {\it outgoing} \ {\it or}  \ {\it incoming} \ {\it radiation} \ {\it  condition}
\end{split}
\right.
\nonumber
\end{equation}
for $\lambda \in (0,\infty) \setminus \mathcal E(H)$. Extending $F \in L^2(\mathcal O)$ to be 0 outside $\mathcal O$, the solution $u$ exists uniquely by Theorem \ref{ExistnceuniquenessHelmholtz}.
The {\it source-to-solution map} is defined by
\begin{equation}
\mathbb U_{\mathcal O, \pm }(\lambda) : F \to u_{\pm}\big|_{\mathcal O} = R(\lambda \pm i0)F\big|_{\mathcal O}.
\nonumber
\end{equation}
Let us assume that for $i = 1, 2$, $\Omega^{(i)}_{ext}$ is flat and
\begin{equation}
\Omega_{ext}^{(1)} = \Omega_{ext}^{(2)}, 
\label{Omega(1)=Omega(2)Y>R}
\end{equation}
more precisely isometric in the Euclidean sense. We observe the incoming and outgoing waves in the infinity of $\Omega^{(i)}_1$. We put the super-script $\, ^{(i)}$ for all notations relevant to $\Omega^{(i)}$. However, due to the assumption (\ref{Omega(1)=Omega(2)Y>R}), we sometimes omit the super-script $\, ^{(i)}$ for operators subordinate to 
$\Omega^{(i)}_1$. We denote the set $\Omega_{ext}^{(1)} = \Omega_{ext}^{(2)}$ as $\Omega_{ext}$. 
Let $S_{11}^{(i)}(\lambda)$ be the (1,1) entry of the S-matrix $S^{(i)}(\lambda)$. 

\begin{theorem}
\label{TheoremSmatrixtoSourcetoSolutionmap}
 Assume that $S^{(1)}_{11}(\lambda) = S^{(2)}_{11}(\lambda)$ for all $\lambda \in (0,\infty)\setminus\{\mathcal E^{(1)}(H^{(1)}) \cup \mathcal E^{(2)}(H^{(2)})$\}. Then,  
$$
\mathbb U^{(1)}_{\mathcal O,\pm}(\lambda) = \mathbb U^{(2)}_{\mathcal O,\pm}(\lambda) \quad
 for \   all \quad \lambda \in (0,\infty)\setminus\{\mathcal E^{(1)}(H^{(1)}) \cup \mathcal E^{(2)}(H^{(2)})\}.
$$
\end{theorem}

\begin{proof}  
 We consider the outgoing case.
Putting
\begin{equation}
 u^{(i)} =R^{(i)}(\lambda +i0)F,\quad v_{\ell}=P_{1,\ell}(u^{(1)}-u^{(2)}),
\label{Defineui=Rlambdai0F}
\end{equation}
we observe the far fields of $P_{1,\ell}u^{(1)}$ and $P_{1,\ell}u^{(2)}$. 

\medskip
\noindent
(I) {\it Spherical wave channels}: 
 Let us consider the case $\lambda > \lambda_{1,\ell}$. 
 Let ${\bf h}_1(\lambda) = L^2(S^{n-1})\times {\rm span}\,\{\varphi_{1,\ell} \, ; \, \lambda_{1,\ell} < \lambda\}$ be as in (\ref{hlambdadefine}).
Take $a_1 \in  {\bf h}_1(\lambda)$ arbitrarily, and put $w^{(i)} = \mathcal F_{+}^{(i)}(\lambda)^{\ast}a$, where $a = (a_1,0,\cdots,0)$, and $\mathcal F_{+}^{(i)}(\lambda)$ is $\mathcal F_{+}(\lambda)$ for $\Omega^{(i)}$. 
We prove that
\begin{equation} 
P_{1,\ell}(w^{(1)} - w^{(2)}) = 0, \quad {\it on} \quad \Omega_{ext}.
\label{Pellw(1)=Pellw2}
\end{equation}
In fact, thanks to (\ref{EigenEquation}),
\begin{equation}
(- \Delta_{G}-\lambda)(w^{(1)}- w^{(2)}) = 0 \quad
{\it on} \quad \Omega_{ext}.
\label{Th2.1.1Deltau1-u2=0}
\end{equation}
Multiplying the projection $P_{1,\ell}$ to (\ref{Th2.1.1Deltau1-u2=0}), we have
\begin{equation}
( - \Delta_y - \lambda + \lambda_{1,\ell})P_{1,\ell}(w^{(1)} - w^{(2)}) = 0.
\label{Deltay-lambdalambda1ellPellw=0}
\end{equation}

We show that
\begin{equation}
 P_{1,\ell}(w^{(1)} - w^{(2)}) \simeq 0 \quad {\it on} \quad \Omega_{ext}.
\label{Pellw(1)=Pellw(2)}
\end{equation}
In fact, as $\mathcal F_{1,\ell,+}(\lambda)^{\ast} = \Big(\chi_1 - R(\lambda + i0)\mathcal V_1\Big)\mathcal F_{1,\ell}^{(0)}(\lambda)^{\ast}$ by (\ref{mathcalFjellpmlambdadefine}), we have  by Lemma \ref{P1lRlambda+i0fasympto},
\begin{equation}
\begin{split}
\mathcal F_{1,\ell,+}(\lambda)^{\ast}a - \chi_1\mathcal F_{1,\ell}^{(0)}(\lambda)^{\ast}a& \simeq C(\lambda) \frac{e^{i|y|\sqrt{\lambda - \lambda_{1,\ell}}}}{|y|^{(n-1)/2}}\mathcal F_{1,+}^{(phy)}(\lambda)\mathcal V_1 \mathcal F_{1,\ell}^{(0)}(\lambda)^{\ast} a\\
& \simeq C(\lambda) \frac{e^{i|y|\sqrt{\lambda - \lambda_{1,\ell}}}}{|y|^{(n-1)/2}}(\lambda)A_{11}(\lambda)P_{1,\ell}a
\end{split}
\nonumber
\end{equation}
for some constant $C(\lambda)$, 
where $A_{11}(\lambda)$ is the scattering amplitude for $\Omega^{(1)} = \Omega^{(2)}$ (cf. (\ref{DefineAjklambda})).
The assumption
 $S^{(1)}_{11}(\lambda) = S^{(2)}_{11}(\lambda)$ then implies 
(\ref{Pellw(1)=Pellw(2)}).

Recall the classical Rellich type theorem  : If $w$ satisfies $( - \Delta - E)w = 0$ for $|y| > R_1 > 0$ for some constants $E > 0, R_1>0$, and $\frac{1}{R}\int_{R_1 < |y|<R}|w(y)|^2dy \to 0$ as $R \to \infty$, there exists $R_2 > 0$ such that $w(y) = 0$ for $|y| >R_2$. By virtue of (\ref{Deltay-lambdalambda1ellPellw=0}) and (\ref{Pellw(1)=Pellw(2)}), we have thus proven (\ref{Pellw(1)=Pellw2}).

 Take $F \in L^2(\Omega^{(1)}_{1})= L^2(\Omega^{(2)}_{1})$ with support in $\mathcal O\subset \Omega_{1}^{(1)}
 = \Omega_{1}^{(2)}$.
Then 
\begin{eqnarray*}
 (\mathcal F_{1,\ell,+}^{(1)}(\lambda)F,a_1) &=&(F,\mathcal F_{1,\ell,+}^{(1)}(\lambda)^*a_1) = (F,P_{1,\ell}w^{(1)})_{L^2(\mathcal O)} \\
 &=& (F,P_{1,\ell}w^{(2)})_{L^2(\mathcal O)} 
 =(F,\mathcal F_{1,\ell,+}^{(2)}(\lambda)^* a_1) \\
 &=& (\mathcal F_{1,\ell, +}^{(2)}(\lambda) F,a_1).
\end{eqnarray*}
 As this holds for all $a_1 \in {\bf h}_1(\lambda) \supset {\rm Range}(\mathcal F_{1,\ell, +}^{(i)})$,
 this implies
\begin{equation}
 \mathcal F_{1,\ell,+}^{(1)}(\lambda) F =\mathcal F_{1,\ell,+}^{(2)}(\lambda) F.
\label{mathcalF1(1)=mathcalF1(2)}
\end{equation}

Let us return to (\ref{Defineui=Rlambdai0F}).
In view of Lemma \ref{LemmaRlambdafexpansion} and (\ref{mathcalF1(1)=mathcalF1(2)}), 
we see that the far fields of $P_{\ell}u^{(1)}$ and $P_{\ell}u^{(2)}$ coincide. Then,
\begin{equation}
 (- \Delta_y - \lambda + \lambda_{1,\ell})v_{\ell} = 0, \quad {\rm in} \quad \Omega_{ext}, \\
\label{C2S1Delta-lambdav=0}
\end{equation}
 and the far field of $v_{\ell}$ in  $\Omega^{(1)}=\Omega^{(2)}$ is zero, so,
 $v_{\ell}\simeq 0$. Then $v_{\ell}=0$ by the Rellich-type Theorem. 

\medskip
\noindent
(II) {\it Exponentially decaying channels}: 
 Assume that  $\lambda_{1,\ell} > \lambda$. 
 We can argue in the same way as above with the difference that the resolvent $(- \Delta_y - (\lambda - \lambda_{1,\ell}))^{-1}$ now decays exponentially. As the non-physical scattering amplitudes also coincide, we then see that $v_{\ell}$ decays faster than $e^{-r\sqrt{\lambda_{\ell}-\lambda}}r^{-(n-1)/2}$. Then by expanding by spherical harmonics and reducing the problem to the radial equation, we see that $v_{\ell} = 0$ near infinity. Therefore, $u^{(1)} - u^{(2)} = 0$ near infinity. Since $(H^{(0)}_0- \lambda)(u^{(1)} - u^{(2)}) = 0$ holds on $|y| > R_0$, we then have $u^{(1)} = u^{(2)}$ on $|y| > R_0$, in particular on $\mathcal O$. 
\end{proof}

\section{The remaining proof}
 If we pass to the BSP, we are in the same situation as in \cite{IKL10}, \S 6, where the problem is reduced to $(\Omega_1 \cap \{|x| < R_1\})\cup \cdots \cup \Omega_N \cup \mathcal K$, each $\Omega_i$ being a cylindrical domain, and the boundary is $\overline{\Omega_1} \cap \{x= R_1\}$, $x \in {\bf R}$. In the present paper, the cylindrical domain $\Omega_i$ is replaced by a slab and the boundary is a band $\overline{\Omega_1} \cap \{|y| = R_1\}$, $y \in {\bf R}^n, n \geq 2$. We can then mimick the BC method in  \cite{IKL10}, \S 6 word by word to get the same conclusion. We have thus completed the proof of Theorem \ref{MainTheorem}.
 
 If we use the source-to-solution map, the arguments in \cite{IL26(2)} work in the same way with a slight change of notation.
We do not repeat the details.

\end{document}